\documentclass[11pt]{article}
\usepackage{amssymb}
\usepackage{amsthm}
\usepackage{amsmath}
\usepackage[dvips]{epsfig}
\usepackage{graphicx}
\usepackage{color}
\usepackage{url}

\newtheorem{theorem}{Theorem}
\newtheorem{lemma}{Lemma}

\newtheorem{proposition}{Proposition}

\newenvironment{proof1}{
    \noindent {\em Proof }}{\hfill$\Box$}

\begin{document}

\title{\bf Short-wave transverse instabilities of line solitons of \\
the 2-D hyperbolic nonlinear Schr\"odinger equation}

\author{D.E. Pelinovsky$^{1,2}$, E.A. Ruvinskaya$^2$, O.A. Kurkina$^2$, B. Deconinck$^{3}$, \\
$^{1}$ {\small \it Department of Mathematics and Statistics, McMaster
University, Hamilton, Ontario, Canada, L8S 4K1 } \\
$^{2}$ {\small \it Department of Applied Mathematics, Nizhny Novgorod State
Technical University, Nizhny Novgorod, Russia }\\
$^{3}$ {\small \it Department of Applied Mathematics, University of Washington
Seattle, WA 98195-3925, USA}
}

\date{\today}
\maketitle

\begin{abstract}
We prove that line solitons of the two-dimensional hyperbolic nonlinear Schr\"odinger
equation are unstable with respect to transverse perturbations of arbitrarily small periods,
{\em i.e.}, short waves. The analysis is based on the construction of Jost functions
for the continuous spectrum of Schr\"{o}dinger operators,
the Sommerfeld radiation conditions, and the Lyapunov--Schmidt decomposition. Precise asymptotic
expressions for the instability growth rate are derived in the limit of short periods.
\end{abstract}

\section{Introduction}

Transverse instabilities of line solitons have been studied in many nonlinear evolution equations
(see the pioneering work \cite{ZakRub} and the review article \cite{KivPel}). In particular, this problem
has been studied in the context of the hyperbolic nonlinear Schr\"{o}dinger (NLS) equation
\begin{equation}
\label{NLS}
i \psi_t + \psi_{xx} - \psi_{yy} + 2 |\psi|^2 \psi = 0,
\end{equation}
which models oceanic wave packets in deep water. Solitary waves of the one-dimensional ($y$-independent)
NLS equation exist in closed form. If all parameters of a solitary wave have been
removed by using the translational and scaling invariance, we
can consider the one-dimensional trivial-phase solitary wave in the simple form $\psi = {\rm sech}(x) e^{it}$.
Adding a small perturbation $e^{i \rho y + \lambda t + i t} (U(x) + i V(x))$
to the one-dimensional solitary wave and linearizing the underlying equations, we obtain
the coupled spectral stability problem
\begin{equation}
\label{problem-1}
(L_+ - \rho^2) U = -\lambda V, \quad (L_- - \rho^2) V = \lambda U,
\end{equation}
where $\lambda$ is the spectral parameter,
$\rho$ is the transverse wave number of the small perturbation,
and $L_{\pm}$ are given by the Schr\"{o}dinger operators
\begin{equation*}
L_+ = -\partial_x^2 + 1 - 6 {\rm sech}^2(x), \quad
L_- = -\partial_x^2 + 1 - 2 {\rm sech}^2(x).
\end{equation*}
Note that small $\rho$ corresponds to long-wave perturbations in the transverse directions, while large $\rho$ corresponds to short-wave transverse perturbations.

Numerical approximations of unstable eigenvalues (positive real part) of the spectral stability problem (\ref{problem-1})
were computed in our previous work \cite{DecPel} and reproduced recently by independent numerical
computations in \cite[Fig. 5.27]{Yang} and \cite[Fig. 2]{Barashenkov}. Fig. 2 from \cite{DecPel}
is reprinted here as Figure \ref{tfig1}. The figure illustrates various bifurcations at $P_a$, $P_b$,
$P_c$, and $P_d$, as well as the behavior of eigenvalues and the continuous spectrum
in the spectral stability problem (\ref{problem-1}) as a function of the transverse wave number $\rho$.

\begin{figure}[tbp]
\begin{center}
\includegraphics[width=70mm,height=50mm]{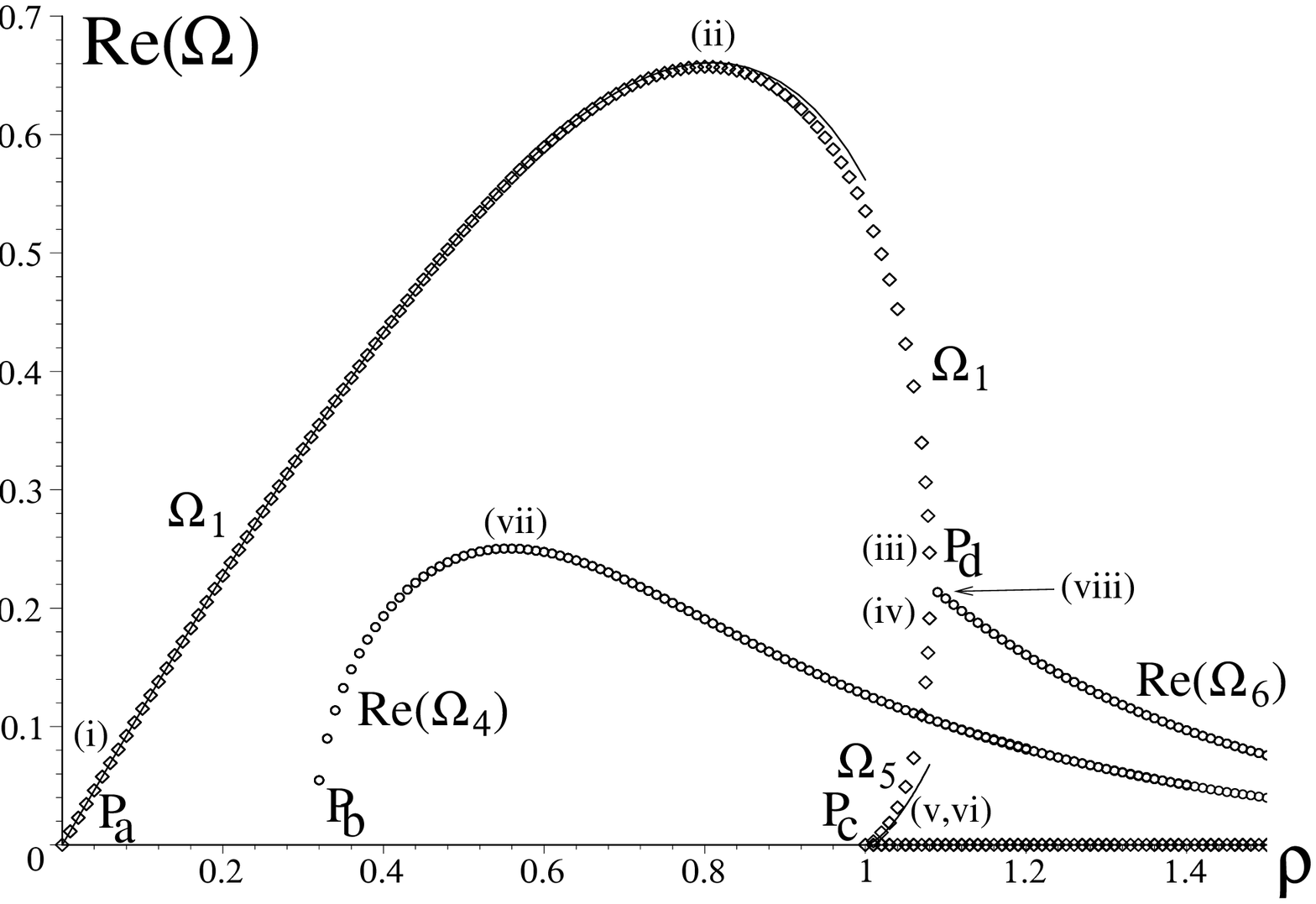}
\includegraphics[width=70mm,height=50mm]{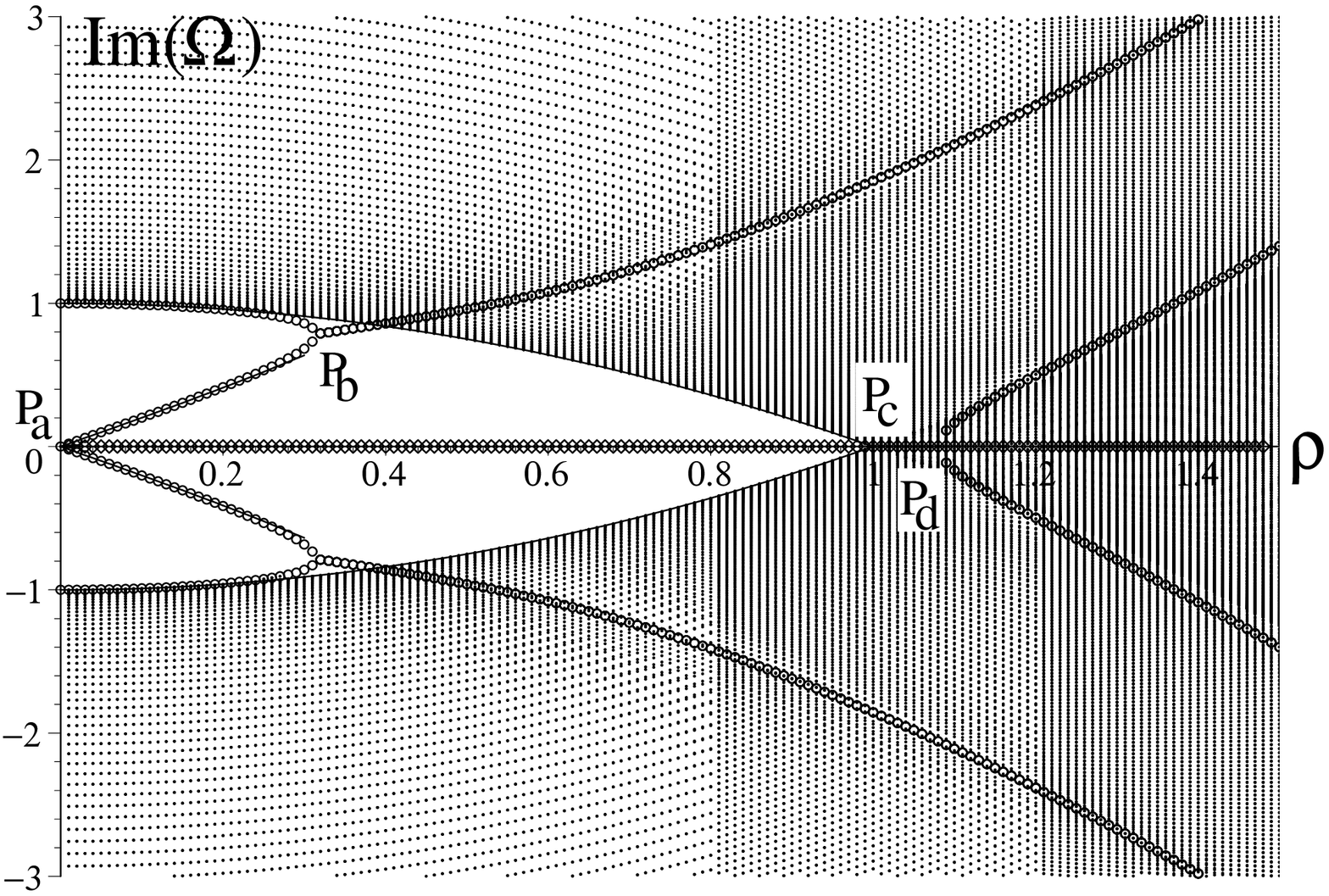}
\end{center}
\caption{Numerical computations of the real (left panel) and imaginary (right panel)
parts of the isolated eigenvalues and the continuous spectrum of the spectral
stability problem (\ref{problem-1}) versus the transverse wave number $\rho$.
Reprinted from \cite{DecPel}.}
\label{tfig1}
\end{figure}

An asymptotic argument for the presence of a real unstable eigenvalue bifurcating at $P_a$
for small values of $\rho$ was given in the pioneering paper \cite{ZakRub}.
The Hamiltonian Hopf bifurcation of a complex quartet at $P_b$ for $\rho \approx 0.31$ was explained in \cite{DecPel}
based on the negative index theory. That paper also proved the bifurcation of a new unstable real eigenvalue
at $P_c$ for $\rho > 1$, using Evans function methods. What is left in this puzzle is an argument for the existence of
unstable eigenvalues for arbitrarily large values of $\rho$. This is the problem
addressed in the present paper.

The motivation to develop a proof of the existence of unstable eigenvalues for large values of $\rho$ originates from different physical experiments (both old and new). First, Ablowitz and Segur \cite{AS79} predicted there are
no instabilities in the limit of large $\rho$ and referred to water wave experiments done in narrow
wave tanks by J. Hammack at the University of Florida in 1979, which showed good agreement with the dynamics
of the one-dimensional NLS equation. Observation of one-dimensional NLS solitons in this limit seems to exclude
transverse instabilities of line solitons.

Second, experimental observations of transverse instabilities are quite robust in the context of
nonlinear laser optics via a four-wave mixing interaction. Gorza {\em et al.}
\cite{Gorza1} observed the primary snake-type instability of line solitons at $P_a$ for small
values of $\rho$ as well as the persistence of the instabilities for large values of $\rho$.
Recently, Gorza {\em et al.} \cite{Gorza2} demonstrated experimentally the presence of the secondary
neck-type instability that bifurcates at $P_b$ near $\rho \approx 0.31$.

In a different physical context of solitary waves in $PT$-symmetric waveguides,
results on the transverse instability of line solitons were
re-discovered by Alexeeva {\em et al.} \cite{Barashenkov}.
(The authors of \cite{Barashenkov} did not notice that their mathematical problem is identical
to the one for transverse instability of line solitons in the hyperbolic NLS equation.)
Appendix B in \cite{Barashenkov} contains asymptotic results suggesting that
if there are unstable eigenvalues of the spectral problem
(\ref{problem-1}) in the limit of large $\rho$, the instability growth rate
is exponentially small in terms of the large parameter $\rho$. No evidence to the fact that
these eigenvalues have {\em nonzero} instability growth rate was reported in \cite{Barashenkov}.

Finally and even more recently, similar instabilities of line solitons
in the hyperbolic NLS equation (\ref{NLS}) were observed numerically
in the context of the discrete nonlinear Schr\"{o}dinger equation
away from the anti-continuum limit \cite{PelYang}.

The rest of this article is organized as follows. Section 2
presents our main results. Section 3 gives the analytical proof of the
main theorem. Section 4 is devoted to computations of the precise asymptotic
formula for the unstable eigenvalues of the spectral stability problem (\ref{problem-1})
in the limit of large values of $\rho$. Section 5 summarizes our findings and discusses further problems.

\section{Main results}

To study the transverse instability of line solitons in the limit of large $\rho$,
we cast the spectral stability problem (\ref{problem-1})
in the semi-classical form by using the transformation
$$
\rho^2 = 1 + \frac{1}{\epsilon^2}, \quad \lambda = \frac{i \omega}{\epsilon^2},
$$
where $\epsilon$ is a small parameter. The spectral problem  (\ref{problem-1})
is rewritten in the form
\begin{equation}
\label{problem-2}
\begin{array}{l}
\left(-\epsilon^2 \partial_x^2 - 1 - 6 \epsilon^2 {\rm sech}^2(x)\right) U = - i \omega V, \\
\left(-\epsilon^2 \partial_x^2 - 1 - 2 \epsilon^2 {\rm sech}^2(x)\right) V = i \omega U.
\end{array}
\end{equation}
Note that we are especially interested in the spectrum of this problem
for $\epsilon\rightarrow 0$, which corresponds to $\rho\rightarrow \infty$
in the original problem. Also, the real part of $\lambda$, which determines
the instability growth rate for (\ref{problem-1}) corresponds, up to a
factor of $\epsilon^2$, to the imaginary part of $\omega$.

Next, we introduce new dependent variables which are more suitable for working with continuous
spectrum for real values of $\omega$:
$$
\varphi := U + i V, \quad \psi := U - i V.
$$
Note that $\varphi$ and $\psi$ are not generally complex conjugates of each other
because $U$ and $V$ may be complex valued since the spectral problem (\ref{problem-2}) is not self-adjoint.
The spectral problem (\ref{problem-2}) is rewritten in the form
\begin{equation}
\label{problem-3}
\begin{array}{l}
\left(-\epsilon^2 \partial_x^2 + \omega - 1 - 4 \epsilon^2 {\rm sech}^2(x)\right) \varphi - 2 \epsilon^2 {\rm sech}^2(x) \psi = 0, \\
\left(-\epsilon^2 \partial_x^2 - \omega - 1 - 4 \epsilon^2 {\rm sech}^2(x)\right) \psi - 2 \epsilon^2 {\rm sech}^2(x) \varphi = 0.
\end{array}
\end{equation}

We note that the Schr\"{o}dinger operator
\begin{equation}
\label{operator-L-0}
L_0 = -\partial_x^2 - 4 {\rm sech}^2(x)
\end{equation}
admits exactly two eigenvalues of the discrete spectrum located at $-E_0$ and $-E_1$ \cite{MF}, where
\begin{equation}
\label{eigenvalues-L-0}
E_0 = \left( \frac{\sqrt{17}-1}{2} \right)^2, \quad
E_1 = \left( \frac{\sqrt{17}-3}{2} \right)^2.
\end{equation}
The associated eigenfunctions are
\begin{equation}
\label{eigenstates-L-0}
\varphi_0 = {\rm sech}^{\sqrt{E_0}}(x), \quad \varphi_1 = \tanh(x) {\rm sech}^{\sqrt{E_1}}(x).
\end{equation}
In the neighborhood of each of these eigenvalues, one can construct a perturbation expansion for
exponentially decaying eigenfunction pairs $(\varphi,\psi)$ and a quartet of complex eigenvalues
$\omega$ of the original spectral problem (\ref{problem-3}). This idea appears already
in Appendix B of \cite{Barashenkov}, where formal perturbation expansions are developed
in powers of $\epsilon$.

Note that the perturbation expansion for the spectral stability problem (\ref{problem-3})
is not a standard application of the Lyapunov--Schmidt reduction method \cite{hale}
because the eigenvalues of the limiting problem given by the operator $L_0$ are embedded
into a branch of the continuous spectrum. Therefore, to justify the perturbation expansions
and to derive the main result, we need a perturbation theory that involves
Fermi's Golden Rule \cite{GS}. An alternative version of this perturbation theory can use
the analytic continuation of the Evans function across the continuous spectrum,
similar to the one in \cite{DecPel}. Additionally, one can think of semi-classical
methods like WKB theory to be suitable for applications to this problem \cite{Dobr}.

The main results of this paper are as follows. To formulate the statements,
we are using the notation $|a| \lesssim \epsilon$ to indicate that for sufficiently small
positive values of $\epsilon$, there is an $\epsilon$-independent positive constant $C$
such that $|a| \leq C \epsilon$. Also, $H^2(\mathbb{R})$ denotes the standard Sobolov space of
distributions whose derivatives up to order two are square integrable.

\begin{theorem}
For sufficiently small $\epsilon > 0$, there exist two quartets of complex
eigenvalues $\{\omega, \bar{\omega}, -\omega, -\bar{\omega}\}$
in the spectral problem (\ref{problem-3}) associated with
the eigenvectors $(\varphi,\psi)$ in $H^2(\mathbb{R})$.

Let $(-E_0,\varphi_0)$ be one of the two eigenvalue--eigenvector pairs of
the operator $L_0$ in (\ref{operator-L-0}). There exists an $\epsilon_0 > 0$
such that for all $\epsilon \in (0,\epsilon_0)$, the complex eigenvalue $\omega$ in the first quadrant
and its associated eigenfunction satisfy
\begin{equation}
|\omega - 1 - \epsilon^2 E_0 | \lesssim \epsilon^3, \quad
\| \varphi - \varphi_0 \|_{L^2} \lesssim \epsilon, \quad
\| \psi\|_{L^{\infty}} \lesssim \epsilon,
\label{bounds-main}
\end{equation}
while the positive value of ${\rm Im}(\omega)$ is exponentially small in $\epsilon$.
\label{theorem-main}
\end{theorem}

\begin{proposition}
Besides the two quartets of complex eigenvalues in Theorem \ref{theorem-main}, no other
eigenvalues of the spectral problem (\ref{problem-3}) exist for sufficiently small $\epsilon > 0$.
\label{proposition-1}
\end{proposition}

\begin{proposition}
The instability growth rates for the two complex quartets of eigenvalues in Theorem \ref{theorem-main}
are given explicitly as $\epsilon \to 0$ by
\begin{equation}
\label{asymptotics}
{\rm Re}(\lambda) = \frac{{\rm Im}(\omega)}{\epsilon^2} \sim
\frac{2^{p+\frac{3}{2}} \pi^2}{[\Gamma(p)]^2} \epsilon^{3-2p} e^{-\frac{\sqrt{2} \pi}{\epsilon}}, \quad
{\rm Re}(\lambda) = \frac{{\rm Im}(\omega)}{\epsilon^2} \sim
\frac{2^{p+\frac{5}{2}} \pi^2}{q^2 [\Gamma(q)]^2} \epsilon^{1-2q} e^{-\frac{\sqrt{2} \pi}{\epsilon}},
\end{equation}
where $p = 2 + \sqrt{E_0}$ and $q = 2 + \sqrt{E_1}$.
\label{proposition-2}
\end{proposition}

Note that the result of Theorem \ref{theorem-main} guarantees that the two quartets of complex eigenvalues
that we can see on Figure \ref{tfig1} remain unstable for all large values of the transverse wave number $\rho$
in the spectral stability problem (\ref{problem-1}).

\section{Proof of Theorem \ref{theorem-main}}

By the symmetry of the problem, we need to prove Theorem \ref{theorem-main} only for one eigenvalue of each
complex quartet, {\em e.g.}, for $\omega$ in the first quadrant of the complex plane.
Let $\omega = 1 + \epsilon^2 E$ and rewrite the spectral problem (\ref{problem-3})
in the equivalent form
\begin{equation}
\label{problem-4}
\begin{array}{l}
\left(-\partial_x^2 - 4 {\rm sech}^2(x)\right) \varphi - 2 {\rm sech}^2(x) \psi = - E \varphi, \\
- 2 \psi - \epsilon^2 \left(\partial_x^2 + E + 4 {\rm sech}^2(x) \right) \psi = 2 \epsilon^2 {\rm sech}^2(x) \varphi.
\end{array}
\end{equation}

At the leading order, the first equation of system (\ref{problem-4}) has exponentially decaying
eigenfunctions (\ref{eigenstates-L-0}) for $E = E_0$ and $E = E_1$ in (\ref{eigenvalues-L-0}).
However, the second equation of system (\ref{problem-4})
does not admit exponentially decaying eigenfunctions for these values of $E$ because the operator
$$
L_{\epsilon}(E) := -2 - \epsilon^2 \left(\partial_x^2 + E + 4 {\rm sech}^2(x) \right)
$$
is not invertible for these values of $E$. The scattering problem for Jost functions
associated with the continuous spectrum of the operator $L_{\epsilon}(E)$ admits
solutions that behave at infinity as
$$
\psi(x) \sim e^{i k x}, \quad \mbox{\rm where} \quad k^2 = E + \frac{2}{\epsilon^2}.
$$
If ${\rm Im}(E) > 0$, then ${\rm Re}(k) {\rm Im}(k) > 0$. The Sommerfeld radiation conditions
$\psi(x) \sim e^{\pm i k x}$ as $x \to \pm \infty$ correspond to solutions $\psi(x)$ that are
exponentially decaying in $x$ when $k$ is extended from
real positive values for ${\rm Im}(E) = 0$ to complex values with ${\rm Im}(k) > 0$
for ${\rm Im}(E) > 0$. Thus
we impose Sommerfeld boundary conditions for the component $\psi$
satisfying the spectral problem (\ref{problem-4}):
\begin{equation}
\label{Sommerfeld}
\psi(x) \to a \left\{ \begin{array}{l} e^{i k x}, \quad \quad \;\; x \to \infty, \\
\sigma e^{-i k x}, \quad x \to -\infty, \end{array} \right. \quad
k = \frac{1}{\epsilon} \sqrt{2 + \epsilon^2 E},
\end{equation}
where $a$ is the radiation tail amplitude to be determined
and $\sigma = \pm 1$ depends on whether $\psi$ is even or odd in $x$.
To compute $a$, we note the following elementary result.

\begin{lemma}
Consider bounded (in $L^\infty(\mathbb{R})$) solutions $\psi(x)$ of the second-order differential equation
\begin{equation}
\label{Green}
\psi'' + k^2 \psi = f,
\end{equation}
where $k \in \mathbb{C}$ with ${\rm Re}(k) > 0$ and ${\rm Im}(k) \geq 0$,
whereas $f \in L^1(\mathbb{R})$ is a given function, either even or odd.
Then
\begin{equation}
\label{wave}
\psi(x) = \frac{1}{2 i k } \int_{-\infty}^{x} e^{ik(x-y)} f(y) dy
+ \frac{1}{2ik} \int_x^{+\infty} e^{-ik(x-y)} f(y)  dy
\end{equation}
is the unique solution of the differential equation (\ref{Green})
with the same parity as $f$ that satisfies the Sommerfeld radiation conditions (\ref{Sommerfeld})
with
\begin{equation}
\label{tail-amplitude}
a = \frac{1}{2ik} \int_{-\infty}^{+\infty} f(y) e^{-iky} dy.
\end{equation}
\label{lemma-continuous}
\end{lemma}

\begin{proof}
Solving (\ref{Green}) using variation of parameters, we obtain
$$
\psi(x) = e^{i k x} \left[ u(0) + \frac{1}{2ik} \int_0^{x} f(y) e^{-iky} dy \right]
+ e^{-i k x} \left[ v(0) -  \frac{1}{2ik} \int_0^{x} f(y) e^{iky} dy \right],
$$
where $u(0)$ and $v(0)$ are arbitrary constants. We fix these constants using the Sommerfeld radiation conditions (\ref{Sommerfeld}), which yields
$$
u(0) =  \frac{1}{2ik} \int_{-\infty}^0 f(y) e^{-iky} dy, \quad
v(0) =  \frac{1}{2ik} \int_0^{+\infty} f(y) e^{iky} dy.
$$
Using these expressions and the definition
$a = \lim_{x \to \infty} \psi(x) e^{-i k x}$, we obtain
(\ref{wave}) and (\ref{tail-amplitude}). It is easily checked that $\psi$ has the same parity as $f$.
\end{proof}

To prove Theorem \ref{theorem-main}, we select
one of the two eigenvalue--eigenvector pairs $(E_0,\varphi_0)$ of
the operator $L_0$ in (\ref{operator-L-0}) and
proceed with the Lyapunov--Schmidt decomposition
$$
E = E_0 + \mathcal{E}, \quad \varphi = \varphi_0 + \phi, \quad \phi \perp \varphi_0.
$$
To simplify calculations, we assume that $\varphi_0$ is normalized to unity in the
$L^2$ norm. The orthogonality condition $\phi \perp \varphi_0$ is used with respect to
the inner product in $L^2(\mathbb{R})$ and $\phi \in L^2(\mathbb{R})$ is assumed in
the decomposition.

The spectral problem (\ref{problem-4}) is rewritten in the form
\begin{equation}
\label{problem-5}
\begin{array}{l}
\left( L_0 + E_0 \right) \phi - 2 {\rm sech}^2(x) \psi = - \mathcal{E} (\varphi_0 + \phi), \\
L_{\epsilon}(E_0+\mathcal{E}) \psi = 2 \epsilon^2 {\rm sech}^2(x) (\varphi_0 + \phi).
\end{array}
\end{equation}
Because $\phi \perp \varphi_0$, the correction term $\mathcal{E}$ is uniquely determined by
projecting the first equation of the system (\ref{problem-5}) onto $\varphi_0$:
\begin{equation}
\label{bifurcation}
\mathcal{E} = 2 \int_{-\infty}^{\infty} {\rm sech}^2(x) \varphi_0(x) \psi(x) dx.
\end{equation}
If $\psi \in L^{\infty}(\mathbb{R})$, then $|\mathcal{E}| = \mathcal{O}(\| \psi \|_{L^{\infty}})$.
Let $P$ be the orthogonal projection from $L^2(\mathbb{R})$ to the range of $(L_0 + E_0)$.
Then, $\phi$ is uniquely determined from the linear inhomogeneous equation
\begin{equation}
\label{inhomogeneous-phi}
P(L_0 + E_0 + \mathcal{E})P \phi = 2 {\rm sech}^2(x) \psi - 2 \varphi_0
\int_{-\infty}^{\infty} {\rm sech}^2(x) \varphi_0(x) \psi(x) dx,
\end{equation}
where $P(L_0 + E_0) P$ is invertible with a bounded inverse and $\psi \in L^{\infty}(\mathbb{R})$
is assumed. On the other hand, $\psi \in L^{\infty}(\mathbb{R})$ is uniquely found using the
linear inhomogeneous equation
\begin{equation}
\label{inhomogeneous-psi}
\psi'' + k^2 \psi = f, \quad {\rm where} \quad
f = - 2{\rm sech}^2(x) (\varphi_0 + \phi + 2 \psi),
\end{equation}
subject to the Sommerfeld radiation condition (\ref{Sommerfeld}), where $\phi \in L^{\infty}(\mathbb{R})$
is assumed. Note that $\psi$ is not real because of the Sommerfeld radiation condition (\ref{Sommerfeld})
and depends on $\epsilon$ because of the $\epsilon$-dependence of $k$ in
\begin{equation}
\label{def-k}
k = \frac{1}{\epsilon} \sqrt{2 + \epsilon^2 E_0 + \epsilon^2 \mathcal{E}}.
\end{equation}
We are now ready to prove Theorem \ref{theorem-main}. \\

\begin{proof1}{\em of Theorem \ref{theorem-main}.}
The function $f$ on the right-hand-side of (\ref{inhomogeneous-psi})
is exponentially decaying as $|x| \to \infty$ if $\phi, \psi \in L^{\infty}(\mathbb{R})$.
From the solution (\ref{wave}), we rewrite the equation into the integral form
\begin{eqnarray}
\nonumber
\psi(x) & = & \frac{i \epsilon}{\sqrt{2 + \epsilon^2 E_0 + \epsilon^2 \mathcal{E}}}
\int_{-\infty}^{x} e^{ik(x-y)} {\rm sech}^2(y) (\varphi_0 + \phi + 2 \psi)(y) dy \\
& \phantom{t} &
+ \frac{i \epsilon}{\sqrt{2 + \epsilon^2 E_0 + \epsilon^2 \mathcal{E}}}
\int_x^{+\infty} e^{-ik(x-y)} {\rm sech}^2(y) (\varphi_0 + \phi + 2 \psi)(y)  dy.
\label{integral}
\end{eqnarray}
The right-hand-side operator acting on $\psi \in L^{\infty}(\mathbb{R})$
is a contraction for small values of $\epsilon$ if $\phi \in L^{\infty}(\mathbb{R})$
and $\mathcal{E} \in \mathbb{C}$ are bounded as $\epsilon \to 0$, and for
${\rm Im}(\mathcal{E}) \geq 0$ (yielding ${\rm Im}(k) \geq 0$). By the Fixed Point Theorem
\cite{hale}, we have a unique solution  $\psi \in L^{\infty}(\mathbb{R})$ of the integral equation
(\ref{integral}) for small values of $\epsilon$ such that
$\| \psi \|_{L^{\infty}} = \mathcal{O}(\epsilon)$ as $\epsilon \to 0$.
This solution can be substituted into the inhomogeneous equation (\ref{inhomogeneous-phi}).

Since $|\mathcal{E}| = \mathcal{O}(\| \psi \|_{L^{\infty}}) = \mathcal{O}(\epsilon)$ as $\epsilon \to 0$
and the operator $P(L_0 + E_0) P$ is invertible with a bounded inverse,
we apply the Implicit Function Theorem and obtain a unique solution $\phi \in H^2(\mathbb{R})$
of the inhomogeneous equation (\ref{inhomogeneous-phi}) for small values of $\epsilon$
such that $\| \phi \|_{H^2} = \mathcal{O}(\epsilon)$ as $\epsilon \to 0$. Note that by Sobolev
embedding of $H^2(\mathbb{R})$ to $L^{\infty}(\mathbb{R})$, the earlier assumption $\phi \in L^{\infty}(\mathbb{R})$
for finding $\psi \in L^{\infty}(\mathbb{R})$ in (\ref{inhomogeneous-psi}) is consistent
with the solution $\phi \in H^2(\mathbb{R})$.

This proves bounds (\ref{bounds-main}). It remains to show that ${\rm Im}(\mathcal{E}) > 0$
for small nonzero values of $\epsilon$. If so, then the real eigenvalue
$1 + \epsilon^2 E_0$ bifurcates to the first complex quadrant and yields the eigenvalue
$\omega = 1 + \epsilon^2 E_0 + \epsilon^2 \mathcal{E}$
of the spectral problem (\ref{problem-3}) with ${\rm Im}(\omega) > 0$.
Persistence of such an isolated eigenvalue with respect to small values of $\epsilon$
follows from regular perturbation theory. Also, the eigenfunction $\psi$ in (\ref{integral})
is exponentially decaying in $x$ at infinity if ${\rm Im}(\mathcal{E}) > 0$. As a result,
the eigenvector $(\phi,\psi)$ is defined in $H^2(\mathbb{R})$
for small nonzero values of $\epsilon$, although $\| \psi \|_{H^2}$ diverges as $\epsilon \to 0$.

To prove that ${\rm Im}(\mathcal{E}) > 0$ for small but nonzero values of $\epsilon$, we
use (\ref{Sommerfeld}) and (\ref{inhomogeneous-psi}), integrate by parts, and
obtain the exact relation
\begin{eqnarray*}
2 \int_{-\infty}^{\infty} {\rm sech}^2(x) (\varphi_0 + \phi) \psi(x) dx & = &
\int_{-\infty}^{\infty} \bar{\psi}(x) \left( - \partial_x^2 - k^2 - 4 {\rm sech}^2(x) \right) \psi(x) dx \\
& = & \left( - \bar{\psi} \psi_x + \bar{\psi}_x \psi \right) \biggr|_{x \to -\infty}^{x \to +\infty} \\
& \phantom{t} & \phantom{text}  +
\int_{-\infty}^{\infty} \psi(x) \left( - \partial_x^2 - k^2 - 4 {\rm sech}^2(x) \right) \bar{\psi}(x) dx \\
& = & 4 i k |a(\epsilon)|^2 + 2 \int_{-\infty}^{\infty} {\rm sech}^2(x) (\varphi_0 + \phi) \bar{\psi}(x) dx.
\end{eqnarray*}
By using bounds (\ref{bounds-main}), definition (\ref{tail-amplitude}), and projection
(\ref{bifurcation}), we obtain
\begin{eqnarray}
\nonumber
{\rm Im}(\mathcal{E}) & = & 2 {\rm Im} \int_{-\infty}^{\infty} {\rm sech}^2(x) \varphi_0(x) \psi(x) dx
= 2 k |a(\epsilon)|^2 \left(1 + \mathcal{O}(\epsilon) \right)\\
\label{im-E}
& = & \frac{2}{k} \left| \int_{-\infty}^{+\infty} {\rm sech}^2(x) \varphi_0(x) e^{-ikx} dx \right|^2
\left(1 + \mathcal{O}(\epsilon) \right),
\end{eqnarray}
which is strictly positive. Note that this expression is referred to as Fermi's Golden Rule
in quantum mechanics \cite{GS}.
Since $k = \mathcal{O}(\epsilon^{-1})$ as $\epsilon \to 0$,
the Fourier transform of ${\rm sech}^2(x) \varphi_0(x)$ at this $k$
is exponentially small in $\epsilon$. Therefore, ${\rm Im}(\omega) > 0$ is exponentially small
in $\epsilon$. The statement of the theorem is proved.
\end{proof1}

\section{Proofs of Propositions \ref{proposition-1} and \ref{proposition-2}}

To prove Proposition \ref{proposition-1}, let us fix $E_c$ to be $\epsilon$-independent and
different from  $E_0$ and $E_1$ in (\ref{eigenvalues-L-0}). We write $E = E_c + \mathcal{E}$
for some small $\epsilon$-dependent values of $\mathcal{E}$.
The spectral problem (\ref{problem-4}) is rewritten as
\begin{equation}
\label{problem-15}
\begin{array}{l}
\left( L_0 + E_c \right) \varphi - 2 {\rm sech}^2(x) \psi = - \mathcal{E} \varphi, \\
L_{\epsilon}(E_c+\mathcal{E}) \psi = 2 \epsilon^2 {\rm sech}^2(x) \varphi.
\end{array}
\end{equation}
\\

\begin{proof1}{\em of Proposition \ref{proposition-1}.}
If $E_c$ is real and negative, the system (\ref{problem-15}) has only
oscillatory solutions, hence exponentially decaying eigenfunctions
do not exist for values of $E$ near $E_c$.
Furthermore, note that the Schr\"{o}dinger operator $L_0$ in (\ref{operator-L-0})
has no end-point resonances. Therefore no bifurcation of isolated eigenvalues
may occur if $E_c = 0$. Thus, we consider positive values of $E_c$
if $E_c$ is real and values with ${\rm Im}(E_c) > 0$ if $E_c$ is complex.

By Lemma \ref{lemma-continuous}, we rewrite the second equation of the system
(\ref{problem-15}) in the integral form
\begin{eqnarray}
\nonumber
\psi(x) & = & \frac{i \epsilon}{\sqrt{2 + \epsilon^2 E_c + \epsilon^2 \mathcal{E}}}
\int_{-\infty}^{x} e^{ik(x-y)} {\rm sech}^2(y) (\varphi + 2 \psi)(y) dy \\
& \phantom{t} &
+ \frac{i \epsilon}{\sqrt{2 + \epsilon^2 E_c + \epsilon^2 \mathcal{E}}}
\int_x^{+\infty} e^{-ik(x-y)} {\rm sech}^2(y) (\varphi + 2 \psi)(y)  dy.
\label{integral-15}
\end{eqnarray}
Again, the right-hand-side operator on $\psi \in L^{\infty}(\mathbb{R})$
is a contraction for small values of $\epsilon$ if $\varphi \in L^{\infty}(\mathbb{R})$
and $\mathcal{E} \in \mathbb{C}$ are bounded as $\epsilon \to 0$, and for
${\rm Im}(E_c + \mathcal{E}) \geq 0$ (yielding ${\rm Im}(k) \geq 0$). By the Fixed Point Theorem,
under these conditions we have a unique solution  $\psi \in L^{\infty}(\mathbb{R})$ of the integral equation
(\ref{integral-15}) for small values of $\epsilon$ such that
$\| \psi \|_{L^{\infty}} = \mathcal{O}(\epsilon)$ as $\epsilon \to 0$.
This solution can be substituted into the first equation of the system
(\ref{problem-15}).

The operator $L_0 + E_c$ is invertible with a bounded inverse
if $E_c$ is complex or if $E_c$ is real and positive but different from $E_0$ and $E_1$.
By the Implicit Function Theorem, we obtain a unique solution $\varphi = 0$
of this homogeneous equation for small values of $\epsilon$ and for any value of
$\mathcal{E}$ as long as $\mathcal{E}$ is small as $\epsilon \to 0$ (since $E_c$ is fixed
independently of $\epsilon$). Next, with $\varphi = 0$, the unique solution
of the integral equation (\ref{integral-15}) is $\psi = 0$, hence $E = E_c + \mathcal{E}$
is not an eigenvalue of the spectral problem (\ref{problem-4}).
\end{proof1}

\vspace{0.2cm}

To prove Proposition \ref{proposition-2}, we compute ${\rm Im}(\omega)$ in
Theorem \ref{theorem-main} explicitly in the asymptotic limit $\epsilon \to 0$.
It follows from (\ref{def-k}) and
(\ref{im-E}) that
$$
{\rm Im}(\omega) =  \sqrt{2} \epsilon^3  \left| \int_{-\infty}^{+\infty} {\rm sech}^2(x) \varphi_0(x)
e^{-i k x} dx \right|^2 \left(1 + \mathcal{O}(\epsilon) \right),
$$
where $k = \sqrt{2} \epsilon^{-1} (1 + \mathcal{O}(\epsilon^2))$.\\

\begin{proof1}{\em of Proposition \ref{proposition-2}.}
Let us consider the first eigenfunction $\varphi_0$ in (\ref{eigenstates-L-0}) for the lowest
eigenvalue in (\ref{eigenvalues-L-0}). Using integral {\bf 3.985} in \cite{Grad}, we obtain
\begin{eqnarray*}
I_0 & = & \int_{-\infty}^{+\infty} {\rm sech}^2(x) \varphi_0(x)
e^{-i k x} dx = 2 \int_0^{\infty} {\rm sech}^p(x) \cos(kx) dx =
\frac{2^{p-1}}{\Gamma(p)} \left| \Gamma\left( \frac{p + i k}{2} \right) \right|^2,
\end{eqnarray*}
where $p = 2 + \sqrt{E_0} = (\sqrt{17}+3)/2$. Since $k = \mathcal{O}(\epsilon^{-1})$ and $\epsilon \to 0$,
we have use the asymptotic limit {\bf 8.328} in \cite{Grad}:
\begin{equation}
\label{limit}
\lim_{|y| \to \infty} |\Gamma(x+iy)| e^{\frac{\pi}{2} |y|} |y|^{\frac{1}{2}-x} = \sqrt{2\pi},
\end{equation}
from which we establish the asymptotic equivalence:
\begin{eqnarray*}
I_0 & = & \frac{2^{p-1}}{\Gamma(p)} \left| \Gamma\left( \frac{p + i k}{2} \right) \right|^2
\sim \frac{2 \pi}{\Gamma(p) k^{1-p}} e^{-\frac{\pi}{2} k} \sim
\frac{2^{\frac{p+1}{2}} \pi}{\Gamma(p)} \epsilon^{1-p} e^{-\frac{\pi}{\sqrt{2} \epsilon}}.
\end{eqnarray*}
Therefore, the leading asymptotic order for ${\rm Im}(\omega)$ is given by
\begin{equation}
\label{formula-1}
{\rm Im}(\omega) \sim \frac{2^{p+\frac{3}{2}} \pi^2}{[\Gamma(p)]^2} \epsilon^{5-2p} e^{-\frac{\sqrt{2} \pi}{\epsilon}}.
\end{equation}

Next, let us consider the second eigenfunction $\varphi_1$ in (\ref{eigenstates-L-0}) for the second
eigenvalue in (\ref{eigenvalues-L-0}). Using integral {\bf 3.985} in \cite{Grad}
and integration by parts, we obtain
\begin{eqnarray*}
I_1 & = & \int_{-\infty}^{+\infty} {\rm sech}^2(x) \varphi_1(x)
e^{-i k x} dx = -\frac{2ik}{q} \int_0^{\infty} {\rm sech}^q(x) \cos(kx) dx =
-\frac{i k 2^{q-1}}{q \Gamma(q)} \left| \Gamma\left( \frac{q + i k}{2} \right) \right|^2,
\end{eqnarray*}
where $q = 2 + \sqrt{E_1} = (\sqrt{17}+1)/2$. Using limit (\ref{limit}), we obtain
\begin{eqnarray*}
I_1 & = & -\frac{i k 2^{q-1}}{q \Gamma(q)} \left| \Gamma\left( \frac{q + i k}{2} \right) \right|^2
\sim - \frac{2 \pi i k}{q \Gamma(q) k^{1-q}} e^{-\frac{\pi}{2} k} \sim
- \frac{i 2^{\frac{p+2}{2}} \pi}{q \Gamma(q)} \epsilon^{-p} e^{-\frac{\pi}{\sqrt{2} \epsilon}}.
\end{eqnarray*}
Therefore, the leading asymptotic order for ${\rm Im}(\omega)$ is given by
\begin{equation}
\label{formula-2}
{\rm Im}(\omega) \sim \frac{2^{p+\frac{5}{2}} \pi^2}{q^2 [\Gamma(q)]^2} \epsilon^{3-2q} e^{-\frac{\sqrt{2} \pi}{\epsilon}}.
\end{equation}
In both cases (\ref{formula-1}) and (\ref{formula-2}), the expression for ${\rm Im}(\omega)$
have the algebraically large prefactor in $\epsilon$ with
the exponent $5 - 2p = 2 - \sqrt{17} < 0$ and $3 - 2q = 2 - \sqrt{17} < 0$. Nevertheless, ${\rm Im}(\omega)$
is exponentially small as $\epsilon \to 0$.
\end{proof1}

\section{Conclusion}

We have proved that the spectral stability problem (\ref{problem-1}) has exactly two
quartets of complex unstable eigenvalues in the asymptotic limit of large transverse wave numbers.
We have obtained precise asymptotic expressions for the instability growth rate
in the same limit.

It would be interesting to verify numerically the validity of our asymptotic results.
The numerical approximation of eigenvalues in this asymptotic limit is a delicate problem
of numerical analysis because of the high-frequency oscillations of the eigenfunctions
for large values of $\lambda$, {\em i.e.}, small values of $\epsilon$, as discussed in \cite{DecPel}. As
we can see in Figure 1, the existing numerical results do not allow us to compare with
the asymptotic results of our work. This numerical problem is left for further studies.\\

{\bf Acknowledgments:} The work of DEP, EAR, and OAK is supported by the Ministry of Education
and Science of Russian Federation (Project 14.B37.21.0868). BD acknowledges support 
from the National Science Foundation of the USA through grant NSF-DMS-1008001.

\end{document}